\documentclass[12pt,oneside,reqno]{amsart}
\usepackage[utf8]{inputenc}
\usepackage[english]{babel}

\usepackage[margin=1in]{geometry}

\usepackage{multicol}

\usepackage{amsmath}
\usepackage{amsthm}
\usepackage{amssymb}
\usepackage{graphicx}
\usepackage{mathrsfs}
\usepackage[colorinlistoftodos]{todonotes}
\usepackage[colorlinks=true, allcolors=black]{hyperref}
\usepackage{comment}
\usetikzlibrary{graphs}
\usepackage{float}
\usepackage{xcolor}
\usepackage{tikz-cd}

\newcommand{\showcomments}{yes}

\newsavebox{\commentbox}
{\ifthenelse{\equal{\showcomments}{yes}}%
{\footnotemark
        \begin{lrbox}{\commentbox}
        \begin{minipage}[t]{1in}\raggedright\sffamily\tiny
        \footnotemark[\arabic{footnote}]}
{\begin{lrbox}{\commentbox}}}%
{\ifthenelse{\equal{\showcomments}{yes}}%
{\end{minipage}\end{lrbox}\marginpar{\usebox{\commentbox}}}
{\end{lrbox}}}

\newtheorem{thm}{Theorem}[section]

\newtheorem{theorem}[thm]{Theorem}
\newtheorem{corollary}[thm]{Corollary}
\newtheorem{lemma}[thm]{Lemma}
\newtheorem{proposition}[thm]{Proposition}

\newtheorem{claim}[thm]{Claim}

\theoremstyle{definition}

\newtheorem{definition}[thm]{Definition}

\theoremstyle{remark}

\newtheorem{historicalremark}[thm]{Historical Remark}
\newtheorem{notation}[thm]{Notation}

\newtheorem{remark}[thm]{Remark}
\newtheorem{example}[thm]{Example}

\newcommand{\nclose}[1]{\ensuremath{\langle\!\langle#1\rangle\!\rangle}}
\newcommand{\dist}{\textup{\textsf{d}}}

\newcommand{\field}[1]{\mathbb{#1}}
\newcommand{\integers}{\ensuremath{\field{Z}}}

\newcommand{\naturals}{\ensuremath{\field{N}}}
\newcommand{\reals}{\ensuremath{\field{R}}}

\title{A cubical Rips construction}

\author{Macarena Arenas}
\address{DPMMS, Centre for Mathematical Sciences, Wilberforce Road, Cambridge, CB3 0WB, UK}
\email{mcr59@dpmms.cam.ac.uk}

\subjclass[2010]{20F06, 20F67}
\keywords{Small Cancellation, Cube Complexes, Hyperbolic Groups}
\thanks{The author was funded by a Cambridge Trust \& Newnham College Scholarship.}
  
\begin{document}

\begin{abstract}
Given a finitely presented group $Q$ and a compact special cube complex $X$ with non-elementary hyperbolic fundamental group, we produce a non-elementary, torsion-free, cocompactly cubulated hyperbolic group $\Gamma$ that surjects onto $Q$, with kernel isomorphic to a quotient of $G=\pi_1 X$ and such that $max\{cd(G),2\}\geq cd(\Gamma)\geq cd(G)-1$.
\end{abstract}

\maketitle

\section{introduction}

The Rips exact sequence, first introduced by Rips in~\cite{Rips82}, is a useful tool for producing examples of groups satisfying combinations of properties that are not obviously compatible.  It works by taking as an input an arbitrary finitely presented group $Q$, and producing as an output a hyperbolic group $\Gamma$ that maps onto $Q$ with finitely generated kernel. The ``output group" $\Gamma$ is crafted by adding generators and relations to a presentation of $Q$, in such a way that these relations create enough ``noise" in the presentation to ensure hyperbolicity. One can then lift pathological properties of $Q$ to (some subgroup of) $\Gamma$. 
For instance, Rips used his construction to produce the first examples of incoherent hyperbolic groups, hyperbolic groups with unsolvable generalised word problem, hyperbolic groups having finitely generated subgroups whose intersection is not finitely generated, and hyperbolic groups containing infinite ascending chains of $r$-generated groups.

The purpose of this work is to present a new variation of the Rips exact sequence. Our main result is:

\begin{theorem}[Theorem~\ref{thm:main}]
Let  $Q$ be a finitely presented group and $G$ be the fundamental group of a compact special (in the sense of~\cite{HaglundWiseSpecial}) cube complex $X$. If $G$ is hyperbolic and non-elementary, then there is a short exact sequence

\begin{equation*}\label{eq:sesintro}
1 \rightarrow N \rightarrow \Gamma \rightarrow Q \rightarrow 1
\end{equation*}

where \begin{enumerate}
\item $\Gamma$ is a hyperbolic, cocompactly cubulated group, 
\item $N \cong G/K$ for some $K < G$,
\item \label{con:3} $max\{cd(G),2\}\geq cd(\Gamma)\geq cd(G)-1$. In particular, $\Gamma$  is torsion-free.

\end{enumerate}

\end{theorem}

\begin{remark}
By Agol's Theorem~\cite{AgolGrovesManning2012}, the group $\Gamma$ obtained in Theorem~\ref{thm:main} is in fact virtually special.
\end{remark}

Many variations of Rips' original construction have been produced over the years by a number of authors, including Arzhantseva-Steenbock~\cite{arzhantseva2014rips}, Barnard-Brady-Dani~\cite{BBD07}, Baumslag-Bridson-Miller-Short~\cite{BBMS00}, Belegradek-Osin~\cite{BOsin08}, Bridson-Haefliger~\cite{BridsonHaefliger}, Bumagin-Wise~\cite{BW05}, Haglund-Wise~\cite{HaglundWiseSpecial}, Ollivier-Wise~\cite{OllivierWise07}, and Wise~\cite{WiseRFRips,Wise98}. Below is a sample of their corollaries:

\begin{itemize}
\item There exist non-Hopfian groups with Kazhdan's Property (T) \cite{OllivierWise07}.
\item Every countable group embeds in the outer automorphism group of a group with Kazhdan's Property (T) \cite{OllivierWise07, BOsin08}.
\item Every finitely presented group embeds in the outer automorphism group of a finitely generated, residually finite group \cite{WiseRFRips}.
\item There exists an incoherent group that is the fundamental group of a compact negatively curved 2-complex \cite{Wise98}.
\item There exist hyperbolic special groups that contain (non-quasiconvex) non-separable subgroups \cite{HaglundWiseSpecial}.
\item Property (T) and property FA are not recursively recognizable among hyperbolic groups \cite{BOsin08}.
\end{itemize}

The groups in Rips' original constructions are cubulable by~\cite{WiseSmallCanCube04}, as are the groups in~\cite{HaglundWiseSpecial}; on the other extreme, the groups produced in~\cite{OllivierWise07} and in~\cite{BOsin08} can have Property (T), so will not be cubulable in general (see~\cite{NibloReeves97}).

A notable limitation of all available Rips-type techniques is that the hyperbolic group $\Gamma$ surjecting onto $Q$ will have cohomological dimension at most equal to $2$. This is unsurprising, since, in a precise sense, ``most'' hyperbolic groups  are $2$-dimensional (see~\cite{Gromov93} and \cite{Ollivier05invitation}). Moreover, examples of hyperbolic groups having large cohomological dimension are scarce: Gromov conjectured in~\cite{Gromov87} that all constructions of high-dimensional hyperbolic groups must utilise number-theoretic techniques, and later on, Bestvina~\cite{BestvinaProblemList} made this precise by asking whether for every $K >0$ there is an $N >0$ such that all hyperbolic groups of (virtual) cohomological dimension $\geq N$ contain an arithmetic lattice of dimension $\geq K$. Both of these questions have been answered in the negative by work of a number of people, including Mosher-Sageev~\cite{MosherSageev97}, Januszkiewicz-Swiatkowski~\cite{JanuszkiewiczSwiatkowskiCoxeter03}, and later on Fujiwara-Manning~\cite{FujMan10}, and Osajda~\cite{Osajda13}, but flexible constructions are still difficult to come by.

Theorem~\ref{thm:main} produces cocompactly cubulated hyperbolic groups containing quotients of arbitrary special hyperbolic groups. 
While our construction particularises to Rips' original result, it can also produce groups with large cohomological dimension. Thus, it serves to exhibit a collection of examples of hyperbolic groups that is new and largely disjoint from that produced by all other Rips'-type theorems. 

Most versions of Rips' construction, including the original, rely on some form of small cancellation. This is what imposes a bound on the dimension of the groups thus obtained:  group presentations are inherently 2-dimensional objects, and one can prove that the presentation complexes associated to (classical and graphical) small cancellation presentations are aspherical.

We rely instead on \emph{cubical} presentations and \emph{cubical} small cancellation theory, which is intrinsically higher-dimensional. Roughly speaking, cubical small cancellation theory considers higher dimensional analogues of group presentations: pieces in this setting are overlaps between higher dimensional subcomplexes, and cubical small cancellation conditions measure these overlaps. This viewpoint allows for the use of non-positively curved cube complexes and their machinery, and has proved fruitful in many contexts, most notably in Agol's proofs of the Virtual Haken and Virtual Fibred Conjectures~\cite{AgolGrovesManning2012, Agol08}, which build on work of Wise~\cite{WiseIsraelHierarchy} and his collaborators~\cite{BergeronWiseBoundary, HaglundWiseAmalgams, HsuWiseCubulatingMalnormal}.

While many groups have convenient cubical presentations, producing these, or proving that they do satisfy useful cubical small cancellation conditions, is difficult in general. Some examples are discussed in~\cite{WiseIsraelHierarchy}, \cite{ArHag2021}, and \cite{jankiewicz2017cubulating}. Other than these, we are not aware of instances where explicit examples of non-trivial cubical small cancellation presentations are given, nor of many results producing families of examples with some given list of properties. This note can be viewed as one such construction, and can be used to produce explicit examples that are of a fundamentally different nature to those already available.
 
\subsection{Structure of the paper}
In Section~\ref{sec:back} we  present the necessary background on hyperbolicity, quasiconvexity and cubical small cancellation theory. 
In Section~\ref{sec:noise} we state and prove Theorem~\ref{thm:cyclic}, which is the main technical result, and also state and prove some auxiliary lemmas. In Section~\ref{sec:main} we give the proof of Theorem~\ref{thm:main}. Finally, in Section~\ref{sec:coho} we review some standard material on the cohomological dimension of groups, and analyse the cohomological dimension of $\Gamma$. 

\subsection*{Acknowledgements} I am grateful to Henry Wilton, Daniel Groves, and the anonymous referee for useful comments and suggestions, and to Sami Douba and Max Neumann-Coto for stylistic guidance.

\section{Background}\label{sec:back}

We utilise the following theorem of Arzhantseva~\cite{Arzhantseva2001}:

\begin{theorem}\label{thm:quasi}
Let $G$ be a non-elementary torsion-free hyperbolic group and $H$ a quasiconvex subgroup of $G$ of infinite index. Then there exist infinitely many $g \in G$ for which the subgroup $\langle H, g\rangle $ is isomorphic to $ H * \langle g \rangle$ and is quasiconvex in $G$.
\end{theorem}

Since cyclic subgroups of a hyperbolic group $G$ are necessarily quasiconvex, one can repeatedly apply Theorem~\ref{thm:quasi} to produce quasiconvex free subgroups of any finite rank:

\begin{corollary}\label{cor:quasi}
If $G$ is a non-elementary torsion-free hyperbolic group, then for every $n \in \naturals$ there exists a quasiconvex subgroup $F_n < G$. 
\end{corollary}

Recall that for a graph $B$, a subgraph $A \subset B$ is \emph{full} if whenever vertices $a_1, a_2 \in A$ are joined by an edge $e$ of $B$, then $e \subset A$. In other words, $A$ is the subgraph of $B$ induced by $A^0$.
A map $X \rightarrow Y$ between cell complexes is \emph{combinatorial} if it maps cells to cells of the same dimension.
An \emph{immersion} is a local injection.

\begin{definition}\label{def:locisom}
A \emph{local isometry} $\varphi: Y \rightarrow X$ between non-positively curved cube complexes is a combinatorial map such that for each $y\in Y^0$ and $x=\varphi(y)$, the induced map $\varphi:link(y)\rightarrow link(x)$ is an injection of a full subgraph. 
\end{definition}

A more visual way to think about local isometries is the following: an immersion $\varphi$ is a local isometry if whenever two edges $\varphi(e), \varphi(f)$ form the corner of a square in $X$, then $e,f$ already formed the corner of a square in $Y$.

A key property of local isometries is that they are $\pi_1$-injective. It is then natural to ask \emph{which} subgroups of the fundamental group of a non-positively curved cube complex can be realised by local isometries. In the setting of non-positively curved cube complexes with hyperbolic fundamental group, one large class of subgroups having this property is the class of quasiconvex subgroups. This is proved in~\cite{HaglundGraphProduct}, and collected in~\cite[2.31 and 2.38]{WiseIsraelHierarchy} as presented below.

\begin{definition}\label{def:super}
A subspace $\widetilde{Y} \subset \widetilde{X}$ is \emph{superconvex} if it is convex and for every bi-infinite geodesic line $L$, if $L \subset N_r(\widetilde{Y})$ for some $r >0$, then $L \subset \widetilde{Y}$. A map $Y \rightarrow X$ is \emph{superconvex} if the induced map between universal covers $\widetilde{Y} \rightarrow \widetilde{X}$ is an embedding onto a superconvex space.
\end{definition}

\begin{proposition}\label{prop:loc} Let $X$ be a compact non-positively curved cube complex with $\pi_1X$ hyperbolic. Let $H < \pi_1X$ be a quasiconvex subgroup and let $C \subset \widetilde X$ be a compact subspace. Then there exists a superconvex $H$-cocompact subspace $\widetilde{Y} \subset \widetilde{X}$ with $C \subset \widetilde{Y}$.
\end{proposition}

\begin{proposition}\label{cor:loci}
Let $X$ be a compact non-positively curved cube complex with $\pi_1X$ hyperbolic. Let $H < \pi_1X$ be a quasiconvex subgroup. Then there exists a local-isometry $Y \rightarrow X$ with $\pi_1 Y = H$. 
\end{proposition}

\subsection{Cubical small cancellation theory} Cubical presentations, cubical small cancellation theory, and many related notions were introduced in~\cite{WiseIsraelHierarchy}. We recall them below.

\begin{definition}
A \textit{cubical presentation} $\langle X|\{Y_i\}  \rangle$ consists of a connected non-positively curved cube complex $X$ together with a collection of local isometries of connected non-positively curved cube complexes $Y_i \overset{\varphi_i} \longrightarrow X$. In this setting, we shall think of $X$ as a ``generator" and of the $Y_i \rightarrow X$ as ``relators". The fundamental group of a cubical presentation is defined as $\pi_1 X/\nclose{\{\pi_1 Y_i\}}$. 

Associated to a cubical presentation $\langle X|\{\varphi_i:Y_i \rightarrow X \}  \rangle$ there is a \emph{coned-off space} $X^*$ obtained from $(X \cup \{Y_i \times [0,1]\})/\{(y_i, 1)\sim\varphi_i(y_i)\}$ by collapsing each $Y_i\times \{0\}$ to a point.  By the Seifert-Van Kampen Theorem, the group $\pi_1 X/\nclose{\{\pi_1 Y_i\}}$ is isomorphic to $\pi_1X^*$. Thus, the coned-off space is a presentation complex of sorts for  $\langle X|\{Y_i\}  \rangle$.
In practice, when discussing cubical presentations, we often have in mind the coned-off space $X^*$, rather than the abstract cubical presentation.
\end{definition}

\begin{remark}
A group presentation $\langle a_1,  \ldots, a_s | r_1, \ldots, r_m \rangle$  can be interpreted cubically  by letting $X$ be a bouquet of \textit{s} circles and letting each $Y_i$ map to the path determined by $r_i$. On the other extreme, for every non-positively curved cube complex $X$ there is a ``free" cubical presentation $X^*=\langle X| \ \rangle$ with fundamental group $\pi_1X=\pi_1X^*$.
\end{remark}

In the cubical setting, there are 2 types of pieces: \emph{wall-pieces} and \emph{cone-pieces}. Cone-pieces are very much like pieces in the classical sense -- they measure overlaps between relators in the presentation. On the other hand, wall-pieces measure the overlaps between cone-cells and \emph{rectangles} (hyperplane carriers) -- wall-pieces are always trivial in the classical case, since the square part of $X^*$ coincides with the $1$-skeleton of the presentation complex.

Precise definitions follow.

\begin{definition}[Elevations] Let $Y \rightarrow X$ be a map and $\hat X \rightarrow X$ a covering map. An \emph{elevation} $\hat Y \rightarrow \hat X$ is a map satisfying
\begin{enumerate}
\item the composition $\hat{Y} \rightarrow Y \rightarrow X$ equals $\hat{Y} \rightarrow \hat X \rightarrow X$, and
\item assuming all maps involved are basepoint preserving, $\pi_1 \hat{Y}$ equals the preimage of $\pi_1 \hat{X}$ in $\pi_1 Y$.
\end{enumerate}
\end{definition}

\begin{notation}
In the entirety of this text, a path $\sigma \rightarrow X$ is assumed to be a combinatorial path mapping to the 1-skeleton of $X$.
\end{notation}

\begin{definition}[Pieces]\label{def:pieces}
Let $\langle X | \{Y_i\} \rangle$ be a cubical presentation.
An \emph{abstract contiguous cone-piece} of $Y_j$ in $Y_i$ is an intersection $\widetilde{Y}_j \cap \widetilde{Y}_i$ where either $i \neq j$ or where $i = j$
but $\widetilde{Y}_j \neq \widetilde{Y}_i$. A \emph{cone-piece} of $Y_j$ in $Y_i$ is a path $p \rightarrow P$ in an abstract contiguous cone-piece of $Y_j$ in $Y_i$.
An \emph{abstract contiguous wall-piece} of $Y_i$ is an intersection $N(H) \cap \widetilde{Y}_i$ where $N(H)$ is
the carrier of a hyperplane $H$ that is disjoint from $\widetilde{Y}_i$. To avoid having to deal with empty pieces, we shall assume that $H$ is dual to an edge with an endpoint on $\widetilde{Y}_i$. A \emph{wall-piece} of $Y_i$ is a path $p \rightarrow P$ in an abstract contiguous wall-piece of $Y_i$\footnote{The ``abstract contiguous cone-piece" and ``abstract contiguous wall-piece" terminology comes from the fact that it is also a priori necessary to consider ``non-contiguous cone-pieces" and ``non-contiguous wall-pieces", however~\cite[3.7]{WiseIsraelHierarchy} shows that one can limit oneself to the analysis of contiguous pieces.}.

A \emph{piece} is either a cone-piece or a wall-piece.
\end{definition}

\begin{remark} In Definition~\ref{def:pieces}, two lifts of a cone $Y$ are considered identical if they differ by an element of $Stab_{\pi_1X}(\widetilde Y)$. This is in keeping with the conventions of classical small cancellation theory, where overlaps between a relator and any of its cyclic permutations are not regarded as pieces. This hypothesis facilitates replacing relators by their proper powers to achieve good small cancellation conditions in some cases. 
\end{remark}

The $C(p)$ and $C'(\frac{1}{p})$ conditions are now defined as in the classical case (making no distinction between the two types of pieces when counting them). Namely:

\begin{definition}
A cubical presentation $X^*$ satisfies the \textit{$C(p)$ small cancellation condition} if no essential closed path $\sigma \rightarrow X^*$ is the concatenation of fewer than $p$ pieces, and the $C’(\frac{1}{p})$ \textit{condition}  if whenever  $\mu  \rightarrow X^*$ is a piece in an essential closed path $\sigma \rightarrow X^*$, then $|\mu|< \frac{1}{p} |\sigma|$, where $|\mu|$ is the distance between endpoints of $\widetilde \mu \subset\ \widetilde X$.
\end{definition}

As in the classical case, if the fundamental group of $X$ in a cubical presentation $X^*= \langle X|\{Y_i\} \rangle$ is hyperbolic, sufficiently good small cancellation conditions lead to hyperbolicity. The following form of~\cite[4.7]{WiseIsraelHierarchy} follows immediately from the fact that a cubical presentation that is $C'(\frac{1}{\alpha})$  for $\alpha \geq 12$ can be endowed with a non-positively curved angling rule that satisfies the short innerpaths condition when $\alpha \geq 14$ (\cite[3.32 and 3.70]{WiseIsraelHierarchy}):

\begin{theorem}\label{thm:hyp}
Let $X^*$ be a cubical presentation satisfying the $C'(\frac{1}{p})$ small cancellation condition for $p \geq \frac{1}{14}$. Suppose $\pi_1 X$ is hyperbolic and $ X^*$ is compact. Then $\pi_1 X^*$ is hyperbolic.  
\end{theorem}

\begin{definition}
A collection  $\{H_1, \ldots , H_r\}$ of subgroups 
of a group $G$ is \emph{malnormal} provided that $H^g_i\cap H_j ={1}$ unless $i=j$ and $g \in H_i$.
\end{definition}

Compactness, malnormality and superconvexity will together guarantee the existence of a uniform bound on the size of both cone-pieces and wall-pieces. This is the content of~\cite[2.40 and 3.52]{WiseIsraelHierarchy}, which we recall below:

\begin{lemma}\label{lem:maln}
Let $X$ be a non-positively curved cube complex with $\pi_1X$ hyperbolic. For $1 \leq i \leq r$, let $Y_i \rightarrow X$ be a
local-isometry with $Y_i$ compact, and assume the collection $\{\pi_1Y_1, \ldots ,\pi_1Y_r\}$ is malnormal.
Then there is a uniform upper bound $D$ on the diameters of intersections
$g\widetilde Y_i \cap h\widetilde Y_j$ between distinct $\pi_1X$-translates of their universal covers in $\widetilde X$.
\end{lemma}

\begin{lemma}\label{lem:strips}
Let $Y$ be a superconvex cocompact subcomplex of a CAT(0) cube
complex $X$. There exists $D\geq 0$ such that the following holds: 
For each $n \geq 0$, if $I_1 \times I_n\rightarrow X$ is a combinatorial strip whose base ${0}\times I_n$ lies in $Y$,
and such that $\dist((0, 0), (0, n))\geq D$,  then $I_1 \times I_n$ lies in $Y$.
\end{lemma}

Recall that a \emph{wallspace} is a set $X$ together with a collection of \emph{walls} $\{W_i\}_{i \in I}=\mathcal{W}$ where $W_i=\{\overleftarrow W_i,\overrightarrow W_i\}$ and $\overleftarrow W_i,\overrightarrow W_i \subset X$ for each $i \in I$, and such that:
\begin{enumerate}
    \item $\overleftarrow W_i\cup \overrightarrow W_i=X$ and
    \item $\overleftarrow W_i\cap \overrightarrow W_i=\emptyset$.
\end{enumerate}
Moreover, $\mathcal{W}$ satisfies a \emph{finiteness property}: For every $p,q \in X$ the number of walls separating $p$ and $q$, denoted by $\#_{\mathcal{W}}(p,q)$, is finite. The $\overleftarrow W_i, \overrightarrow W_i$ above are the \emph{half-spaces} of $W_i$.

Once we have specified a cubical presentation, we will cubulate its fundamental group $\pi_1X^*$ via Sageev's construction, which produces a CAT(0) cube complex that is dual to a wallspace. We will assume the reader is familiar with this procedure. Good references include~\cite{Sageev95, ChatterjiNiblo04, GGTbook14}. Cocompactness of the action on the dual cube complex will readily follow from Proposition~\ref{prop:coco}, which is a well-known result of Sageev (\cite{Sageev97}). Properness is more delicate, and will follow from Theorem~\ref{thm:proper} once we know that $\pi_1X^*$ is hyperbolic, since in that case $\pi_1X^*$ has no infinite torsion subgroups. 

\begin{proposition}\label{prop:coco} Let $G$ be hyperbolic and $\{H_1, \ldots, H_n\}$ be a collection of quasiconvex subgroups. Then the action of $G$ on the dual CAT(0) cube complex $C$ is cocompact.
\end{proposition}

Before stating Theorem~\ref{thm:proper}, we need some definitions.

\begin{definition}
Let $Y \rightarrow X$ be a local isometry. $Aut_X(Y)$ is the group of automorphisms $\psi: Y \rightarrow Y$ such that the diagram below is commutative:
\[\begin{tikzcd}
Y \arrow[r, "\psi"] \arrow[rd] & Y \arrow[d] \\
                               & X          
\end{tikzcd}\]
If $Y$ is simply connected, then $Aut_X(Y)$ is equal to $Stab_{\pi_1X}(Y)$.
\end{definition}

\begin{definition}\label{def:b6}  A cubical presentation $\langle X | {Y_i} \rangle $ satisfies
the B(6) condition if it satisfies the following conditions:

\begin{enumerate}

\item \label{item:b6.1}(Small Cancellation) $\langle X | {Y_i} \rangle$ satisfies the $C'(\frac{1}{\alpha})$ condition for $\alpha \geq 14$.
\item \label{item:b6.2} (Wallspace Cones) Each $Y_i$ is a wallspace where each wall in $Y_i$ is the union $\sqcup U_j$ of a collection of disjoint embedded 2-sided hyperplanes in $Y_i$, and there is an embedding $\sqcup N(U_j) \rightarrow Y_i$ of the disjoint union of their carriers into $Y_i$. Each such collection separates $Y_i$. Each hyperplane in $Y_i$ lies in a unique wall.
\item \label{item:b6.3} (Hyperplane Convexity) If $P \rightarrow Y_i$ is a path that starts and ends on vertices
on 1-cells dual to a hyperplane $U$ of $Y_i$ and $P$ is the concatenation of at most 7 pieces, then $P$ is path homotopic
in $Y_i$ to a path $P\rightarrow N(U) \rightarrow Y_i$.
\item \label{item:b6.4} (Wall Convexity) Let $S$ be a path in $Y_i$ that starts and ends with 1-cells
dual to the same wall of $Y_i$. If $S$ is the concatenation of at most 7 pieces, then $S$ is path-homotopic into the carrier of a hyperplane of that wall.
\item \label{item:b6.5} (Equivariance) The wallspace structure on each cone $Y$ is preserved by $Aut_X(Y)$. 
\end{enumerate}
\end{definition}

\begin{historicalremark} In the setting of classical small cancellation theory, the $B(2n)$ condition was defined in~\cite{WiseSmallCanCube04}. Specifically, the ``classical" $B(2n)$ condition states that for
each $2$-cell $R$ in a 2-complex $X$, and for each path $S \rightarrow \partial R$ which is the concatenation of at
most $n$ pieces in $X$, we have $|S| \leq \frac{1}{2} |\partial R|$. The classical $B(2n)$ condition is intermediate to the $C'(\frac{1}{2n})$ and $C(2n)$ conditions in the sense that $C'(\frac{1}{2n}) \Rightarrow B(2n) \Rightarrow C(2n)$. While not a perfect parallel with the notion considered here, the notation is meant to suggest the fact that, in the classical setting, the $B(6)$ condition is sufficient to guarantee the existence of a wallspace structure on $X$ that leads to cocompact cubulability.
\end{historicalremark}

The B(6) condition is extremely useful because it facilitates producing a wallspace structure on the coned-off space $X^*$ by starting \emph{only} with a wallspace structure (satisfying some extra conditions) on each of the cones. This is done by defining an equivalence relation $\sim$ on the hyperplanes of $\widetilde X^*$ as explained below.

\begin{definition}\label{def:eqrel} Let $U$ and $U'$ be hyperplanes in $\widetilde X^*$. Then $U \sim U'$  provided that for some translate of some cone $Y_i$ in $\widetilde X^*$, the intersections $U \cap Y_i$ and $U' \cap Y_i$ lie in the same wall of $Y_i$. A \emph{wall} of $\widetilde X^*$ is a collection of hyperplanes of $\widetilde X^*$ corresponding to an equivalence class.
\end{definition}

That this equivalence relation does in fact define a wallspace structure on $X^*$ when the B(6) condition is satisfied is the content of~\cite[5.f]{WiseIsraelHierarchy}.

\begin{definition}
A hyperplane $U$ is \emph{$m$-proximate} to a $0$-cube $v$ if there is a path $P =P_1, \ldots, P_m$
such that each $P_i$ is either a single edge or a piece, $v$ is the initial
vertex of $P_1$ and $U$ is dual to an edge in $P_m$. A wall is \emph{$m$-proximate} to $v$ if it
has a hyperplane that is $m$-proximate to $v$. A hyperplane is \emph{$m$-far} from a $0$-cube if it is not $m'$-proximate to it for any $m'\leq m$.
\end{definition}

\begin{definition} A hyperplane $U$ of a cone $Y$ is \emph{piecefully convex} if the following holds: For any path $\tau\rho \rightarrow Y$ with endpoints on $N(U)$, if $\tau$ is a geodesic and $\rho$ is trivial or lies in a piece of $Y$ containing an edge
dual to $U$, then $\tau\rho$ is path-homotopic in $Y$ to a path $\mu \rightarrow N(U)$.
\end{definition}

The following is remarked upon in~\cite[5.43]{WiseIsraelHierarchy}. We write $\widetilde N(U):=\widetilde{N(U)}$.

\begin{proposition}\label{prop:piececon} Let $K$ be the maximal diameter of any piece of $Y_i$ in $X^*$. Then a hyperplane $U$ of $Y_i$ is piecefully convex provided that its carrier $N(U)$ satisfies: $\dist_{\widetilde Y_i}(g\widetilde N(U), \widetilde N(U))> K$ for any translate $g\widetilde N(U)\neq \widetilde N(U) \subset \widetilde Y_i$.
\end{proposition}

\begin{definition}[Cut by a wall]
Let $g \in G$ be an element acting on $\widetilde X$. An \emph{axis} $\reals_g$ for $g$ is a $g$-invariant copy of $\reals$ in $\widetilde X$. An element $g$ is \emph{cut by a wall} $W$  if $g^nW \cap \reals_g=\{n\}$ for all $n \in \integers$.
\end{definition}

The theorem below is a restatement of~\cite[5.44]{WiseIsraelHierarchy}, together with~\cite[5.45]{WiseIsraelHierarchy}, and the fact that the short innerpaths condition is satisfied when $C'(\frac{1}{\alpha})$ holds for $\alpha \geq 14$.

\begin{theorem}\label{thm:proper}

 Suppose $X^*= \langle X | \{Y_i\} \rangle $ satisfies the following hypotheses:
 \begin{enumerate}
 \item \label{item:prop.1} $X^*$ satisfies the B(6) condition.
\item \label{item:prop.2} Each hyperplane $U$ of each cone $Y_i$ is piecefully convex.
\item \label{item:prop.3} Let $k \rightarrow Y \in \{Y_i\}$ be a geodesic with endpoints $p, q$. Let $U_1$ and $U'_1$ be distinct
hyperplanes in the same wall $w_1$ of $Y$. Suppose $k$ traverses a 1-cell dual to $U_1$, and either $U'_1$ is 1-proximate to $q$ or $k$ traverses a 1-cell dual to $U'_1$. Then there is a wall $w_2$ in $Y$ that separates $p, q$ but is not
2-proximate to $p$ or $q$.
\item \label{item:prop.4} Each infinite order element of $Aut(Y_i)$ is cut by a wall.
 \end{enumerate}

Then the action of $\pi_1X^*$ on $C$ has torsion stabilizers.
\end{theorem}

\section{Cubical noise}\label{sec:noise}

In the classical setting, there are two essentially distinct strategies for producing group presentations satisfying good small cancellation conditions: taking large enough powers of the relators, and adding ``noise" to the presentation by multiplying each relator by a sufficiently long, suitably chosen word. In the cubical setting, taking powers of relators translates to taking finite-degree covers of the cycles that represent the relators, and this method generalises to taking finite-degree covers of cube complexes with more complicated fundamental groups.  This is the line of inquiry that has been most explored, and for which there exist useful theorems producing cubical small cancellation. This will, however, not be suitable for our applications, because once a cubical presentation $\langle X | \{\hat{Y}_i \rightarrow X\} \rangle$ has been obtained by taking covers, any modifications of the cones (other than taking further covers) will dramatically affect the size of pieces, and possibly invalidate whatever  small cancellation conclusions had been attained. Thus, we instead prove a cubical small cancellation theorem that builds on the idea of adding noise to a presentation. The procedure we describe will be more stable, in the sense that slightly perturbing the choice of cones will not affect the small cancellation conclusions. 

\begin{remark}\label{rmk:special}
We state and prove Theorem~\ref{thm:cyclic} in more generality than that which is needed for later applications. In practice the reader may take $Y$ to be a bouquet of finitely many circles, as this is all that is required for the proof of Theorem~\ref{thm:main}. 
It is also worth noting that while the statement of Theorem~\ref{thm:cyclic}~\eqref{item: special} requires that $X$ and $Y$ be special, the proof only uses that hyperplanes are embedded and 2-sided.
\end{remark}

\begin{theorem}\label{thm:cyclic}
Let $X$ and $Y$ be compact non-positively curved cube complexes with hyperbolic fundamental groups and let  $\mathcal{H}$ be the set of hyperplanes of $X$. Let $\{H_1, \ldots , H_r\}$ be a malnormal collection of free, non-abelian quasiconvex
subgroups of $\pi_1X$, and suppose that $H_i \cap Stab(\widetilde{U})$ is trivial or equal to $H_i$ for all $U  \in \mathcal{H}$.
Let $\gamma_1 \rightarrow X\vee Y, \ldots,\gamma_r \rightarrow X\vee Y$ be closed essential paths based at the wedge point and let $y_1, \ldots, y_r$ be the words in $\pi_1X\ast \pi_1 Y$ represented by the $\gamma_i$'s.
Then for each $\alpha \geq 1$ there are cyclic subgroups $\langle w_i \rangle \subset H_i\ast \langle y_i \rangle$ such that $w_i=w'_iy_i$ where $w'_i \in H_i$ for each $i \in \{1, \ldots r\}$ and:

\begin{enumerate}
 \item The group $\pi_1X\ast \pi_1 Y \big/\langle\langle  w_1, \ldots,  w_r \rangle\rangle$ has a cubical presentation satisfying the $C'(\frac{1}{\alpha})$ condition.
 \item If $\alpha \geq 14$, then the group $\pi_1X \ast \pi_1Y \big/\langle\langle w_1, \ldots,  w_r \rangle\rangle$ is hyperbolic. 
 \item \label{item: special} If $X$ and $Y$ are special, there is an $\alpha_0\geq 14$ such that if $\alpha \geq \alpha_0$, then the group $\pi_1X \ast \pi_1 Y \big/\langle\langle  w_1, \ldots,  w_r \rangle\rangle$ acts properly and
cocompactly on a CAT(0) cube complex.
\end{enumerate}

\end{theorem}

\begin{remark}
The reader might wish to compare Theorem~\ref{thm:cyclic} with~\cite[5.48]{WiseIsraelHierarchy}, which is the analogous result for finite-degree coverings, and whose proof informs the proof below. 
\end{remark}

\begin{proof}
\textbf{Obtaining small cancellation:}
Since $H_1, \ldots , H_r$ are quasiconvex subgroups of $\pi_1 X \ast \pi_1 Y$ and $\pi_1X \ast \pi_1 Y$ is hyperbolic, 
then Corollary~\ref{cor:loci} implies that there are based local isometries $C_1 \rightarrow X\vee Y, \ldots, C_r \rightarrow X\vee Y$ of superconvex subcomplexes with $\pi_1 C_i \cong H_i $ for each $i \in \{1, \ldots, r\}$. 
So there is a cubical presentation $(X\vee Y)^* =\langle X \vee Y| \{C_i \rightarrow X\vee Y\}^r_{i=1} \rangle$ with fundamental group $\pi_1X\ast \pi_1 Y \big/\langle\langle \pi_1C_1 ,\ldots, \pi_1C_r \rangle\rangle$.  
By Lemma~\ref{lem:maln} and Lemma~\ref{lem:strips}, malnormality of $\{H_1, \ldots , H_r\}$ and superconvexity of the $\{C_1, \ldots, C_r\}$  ensures that there is a uniform upperbound $K$ on the diameter of pieces. 

By hyperbolicity, any cyclic subgroup of $\pi_1 X \ast \pi_1 Y$ is quasiconvex. So for any choice of cyclic subgroups $\langle w_i \rangle < \pi_1 X \ast \pi_1 Y$ with $i \in \{1, \ldots, r\}$ there are local isometries $W_i \rightarrow X \vee Y$ with $\pi_1 W_i \cong \langle w_i\rangle$. 

Let $\alpha \geq 1$, and choose each $\sigma_i \rightarrow X\vee Y$ so that $\sigma_i=\sigma'_i\gamma_i$  and  satisfying:
\begin{enumerate}
\item  $\sigma'_i$ is a based closed path in $C_i \subset X$, 
\item $\sigma'_i$ is not a proper power, and does not contain subpaths of length $\geq K\alpha$ that are proper powers,
\item $\sigma'_i\gamma_i$ does not have any backtracks,
\item  the  $W_i$ corresponding to $\langle w_i \rangle =\langle \sigma_i \rangle$ has diameter $||W_i||\geq K\alpha^2$.
\end{enumerate} 

For instance, one can choose $\sigma'_i$ to be of the form $\sigma'_i=\lambda_1 \lambda_2\lambda_1 \lambda^2_2\ldots \lambda_1\lambda^{K\alpha }_2 $ where $\lambda_1, \lambda_2$ are paths representing distinct generators of the corresponding $H_i< \pi_1 X$.  Without loss of generality, we can assume that the $\lambda_1, \lambda_2$ and the  $\gamma_1, \ldots,\gamma_r$ have minimal length in their homotopy classes, and therefore that none of the $\lambda_i$ or $\gamma_i$ have any backtracks, so any backtracks in $\sigma'_i\gamma_i$ arise from cancellation between $\sigma'_i$ and $\gamma_i$. If any cancellation happens, we can rechoose $\lambda_1$ and $\lambda_2$ to eliminate it (for instance, by shortening $\lambda_1$ and $\lambda_2$).

Pieces in each $C_i$ have size bounded by $K$, and each $W_i \rightarrow X \vee Y$ factors through the corresponding $C_i$, so the size of pieces between different cone-cells or between cone-cells and rectangles is bounded by $K$, 
the size of pieces between a cone-cell and itself is bounded by $K\alpha-1$ and $||W_i|| \geq (K\alpha)!+K\alpha \geq K\alpha^2$, so $\langle X\vee Y | \{W_i \rightarrow X\vee Y\}^r_{i=1} \rangle$ satisfies the $C'(\frac{1}{\alpha})$ condition.
 
\textbf{Obtaining hyperbolicity:} As explained above, the $\langle w_i \rangle$ can be chosen so that $(X\vee Y)^*=\langle X \vee Y | \{W_i \rightarrow X\vee Y\}^r_{i=1} \rangle$ satisfies the $C'(\frac{1}{\alpha})$ condition for $\alpha \geq 14$. Since $\pi_1X \ast \pi_1 Y$ is hyperbolic and $(X\vee Y)^*$ is compact, Theorem~\ref{thm:hyp} then implies that $(X\vee Y)^*$ is hyperbolic.

\textbf{Obtaining cocompact cubulability:} Define a wallspace structure on $(X\vee Y)^*$ as follows. Firstly, by subdividing $X\vee Y$, we may assume that each $W_i$ has an even number of hyperplanes cutting the generator of $\langle w_i \rangle$. The specialness hypothesis ensures that all hyperplanes of $X\vee Y$ and of each $W_i$ are embedded and $2$-sided, moreover, since each $W_i$ has cyclic fundamental group and $H_i \cap Stab(\widetilde{U})$ is trivial or equal to $Stab(\widetilde{U})$ for each $\widetilde{U}$, all the hyperplanes of each $W_i$ are contractible or have the homotopy type of a circle representing the generator of $\pi_1 W_i$. Hence, we can define a wallspace structure on each of the cones by defining a wall to be either a single hyperplane if the hyperplane does not cut the generator of the corresponding $\langle w_i \rangle$, or by defining a wall to be an equivalence class consisting of two antipodal hyperplanes cutting the generator.  Concretely, if the generator of $\langle w_i \rangle$ is a cycle $\sigma \rightarrow X\vee Y$ of length $2n$, then letting $\sigma=e_1 \ldots e_{2n}$, hyperplanes $U$ and $U'$ are in the same equivalence class if and only if $U$ is dual to $e_j$ and $U'$ is dual to $e_{j+n}$ (mod $n$) for some $j \in \{1, \ldots, 2n\}$. These choices are exemplified in Figure~\ref{fig:conecell}.

We now check condition~\eqref{item:prop.1} of Theorem~\ref{thm:proper}, which will allow us to extend the wallspace structure on the cones to a wallspace structure for  $(X\vee Y)^*$ using the equivalence relation in Definition~\ref{def:eqrel}.

Choose the  $\langle w_i \rangle$ so that $(X\vee Y)^*=\langle X\vee Y | \{W_i \rightarrow X\vee Y\}^r_{i=1} \rangle$ satisfies the $C'(\frac{1}{\alpha})$ condition for $\alpha \geq 14$. With the choice of walls described above, each cone is a wallspace satisfying condition~\eqref{item:b6.2} of Definition~\ref{def:b6}. 
The $C'(\frac{1}{14})$ condition is also sufficient to ensure that condition~\eqref{item:b6.3} is met. Indeed, the only way for a path with endpoints on the carrier of a hyperplane $U$ to not be homotopic into the carrier of the hyperplane is if the path is homotopic into a power of the generator of $W_i$, and such a path would have to traverse at least $14$ pieces.
Moreover, condition~\eqref{item:b6.4} is met by rechoosing the cyclic subgroups so that the cubical presentation satisfies $C'(\frac{1}{16})$. To wit, since pairs $U,U'$ of hyperplanes lying on the same wall $W$ are antipodal and $(X\vee Y)^*$ satisfies $C'(\frac{1}{\alpha})$, the number of pieces in a path $\sigma \rightarrow C_i$ with endpoints on distinct hyperplanes of $W$ is at least $\frac{\alpha}{2}$, so choosing $\alpha \geq 16$ ensures that such a path traverses at least 8 pieces. 
The choice of wallspace on each cone also ensures that condition~\eqref{item:b6.5} is met: any automorphism of $X\vee Y$ sends a wall not cutting a generator to a wall not cutting a generator, and a wall cutting a generator to a wall cutting a generator. 

Thus, condition~\eqref{item:prop.1} of Theorem~\ref{thm:proper} is satisfied.  Since each wall arises from a quasiconvex subgroup, Proposition~\ref{prop:coco} ensures cocompactness of the action on the dual cube complex.
To ensure properness of the action, we check the rest of the conditions of Theorem~\ref{thm:proper}.

Similar modifications to $X^*$ will ensure that conditions~\eqref{item:prop.2} and~\eqref{item:prop.3} of Theorem~\ref{thm:proper} are met. For condition~\eqref{item:prop.2}, by  Proposition~\ref{prop:piececon}, it suffices to ensure that $\dist_{\widetilde W_i}(g\widetilde N(U), \widetilde N(U))> K$ for any translate $g\widetilde N(U)\neq \widetilde N(U) \subset \widetilde W_i$. Since each piece of $W_i$ contains at least 1 edge, this can be guaranteed by rechoosing the $w_i$'s so that $X^*$ satisfies the $C'(\frac{1}{K'})$ condition, where $K'=\max\{K, 16\}$.
Condition~\eqref{item:prop.3} also follows, because any two hyperplanes in the same wall are at least $8$-far, so there is a hyperplane $V$ crossing the generator of $\langle w_i\rangle$ that is $2$-far from both $U$ and from $U'$, and one can ensure that the antipodal hyperplane $V'$ is also $2$-far from both $U$ and from $U'$ by rechoosing $X^*$ so that it satisfies the $C'(\frac{1}{K''})$ condition, where $K''=\max\{2K, 16\}$.

 Finally, the choice of walls implies that condition~\eqref{item:prop.4} also holds: since $\pi_1(W_i)$ is cyclic for each $i \in I$, every element $g \in Aut(W_i)$ has an axis, which is cut by a wall of $X^*$ crossing the generator of $\pi_1(W_i)$. 
\end{proof}

\begin{definition}[Height]
The \emph{height} of $H \leq G$ is the maximal $n \in \naturals$ such that there exist distinct cosets $g_1H, \ldots, g_nH$ for which $H^{g_1}\cap \ldots \cap H^{g_n}$ is infinite.
\end{definition}

In~\cite{GMRS98} it was proven that:

\begin{theorem}\label{thm:finiteheight}
Quasiconvex subgroups of hyperbolic groups have finite height.
\end{theorem}

\begin{definition}
The \emph{commensurator} $C_G(H)$ of a subgroup $H$ of $G$ is the set
 $C_G(H)=\{g \in G: [H :H^g \cap H] < \infty\}$.
\end{definition}

\begin{figure}
\centerline{\includegraphics[scale=0.28]{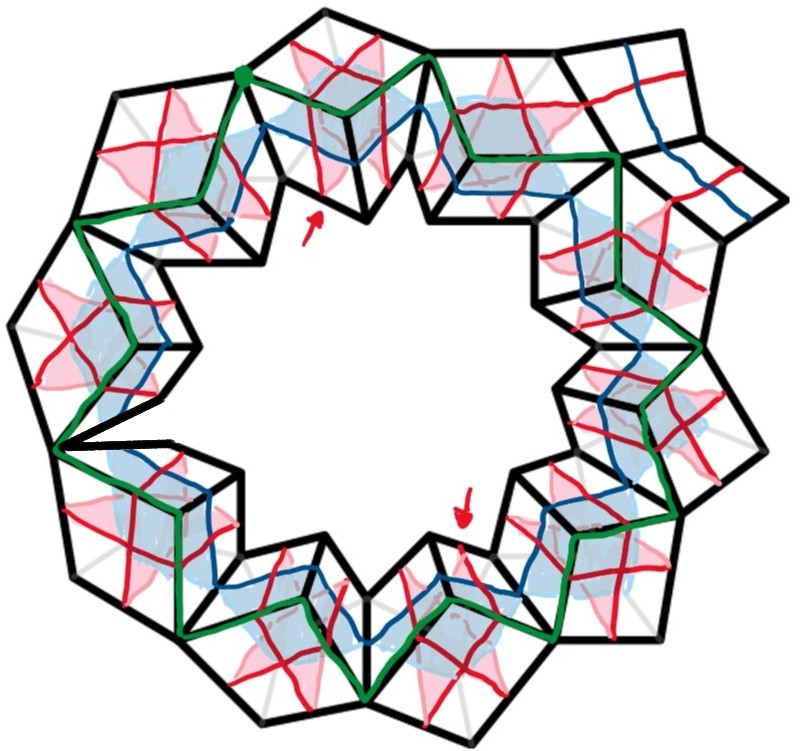}}
\caption{A potential cone-cell and its hyperplanes. The generator of its fundamental group is drawn in green, the hyperplanes that cross it are drawn in red, and the hyperplanes that do not cross it are drawn in blue. A pair of antipodal hyperplanes is indicated.}
\label{fig:conecell}
\end{figure} 

\begin{remark}
If $G$ is hyperbolic and the subgroup $H$ is infinite and quasiconvex, then $[C_G(H) :H]< \infty$ by~\cite{KapovichShort96}. In particular, $C_G(H)$ is also a quasiconvex subgroup of $G$, and if $G$ is torsion-free and $H$ is free and non-abelian then so is $C_G(H)$.
\end{remark}

The following result will be used in the proof of Lemma~\ref{lem:cylics}:

\begin{lemma}\label{lem:comms}\cite[8.6]{WiseIsraelHierarchy}
Let $G$ be hyperbolic and torsion-free and $H_1, \ldots, H_r$ be a collection of quasiconvex subgroups of $G$. Let $K_1, \ldots , K_s$ be representatives of the finitely many distinct
conjugacy classes of subgroups consisting of intersections of collections of
distinct conjugates of $H_1, \ldots, H_r$ in $G$ that are maximal with respect to having
infinite intersection. Then $\{C_G(K_1), \ldots ,C_G(K_s)\}$ is a malnormal
collection of subgroups of $G$.
\end{lemma}

The ensuing lemma surely exists in some form in the literature, but we include a proof for completeness:

\begin{lemma}\label{lem:cylics}
Let $G$ be a non-elementary, torsion-free hyperbolic group. For every $k \in \naturals$, $G$ contains a malnormal collection $\{H_1, \ldots, H_k\}$ of infinite-index quasiconvex non-abelian free subgroups.
\end{lemma}

Before proving Lemma~\ref{lem:cylics}, we make a few observations about malnormal subgroups of free groups. Recall that a subgroup $H <G$ is \emph{isolated} if whenever $g^n \in H$ for $g \in G$, then $g \in H$; a subgroup $H <G$ is \emph{malnormal on generators} if for any generating set $\{ a_1, \ldots, a_n \}$ for $H$ and $g \in G$ and $ g \notin H$, then $a_i^g \notin H$ for any $i \in \{1, \ldots, n\}$.

\begin{lemma}\label{lem:freemaln}\cite[Lemma 1]{FMR02}
Let $F$ be a free group and $H \subset F$ a 2-generator subgroup. Then $H$ is malnormal if and only if $H$ is isolated and malnormal on generators.
\end{lemma}

\begin{claim}\label{clm:separatingcyclics}
Let $F$ be a finite rank non-abelian free group and let $\{h_1, \ldots, h_k\}$ be a finite collection of non-trivial elements of $F$. Then there is a subgroup $J < F$ for which $\{J , \langle h_i \rangle\}$ is a malnormal collection for each $i \in \{1, \ldots, k\}$.
\end{claim}

\begin{proof}
Assume that a basis is given for $F$, and, abusing notation, write also $h_1, \ldots , h_k$ to denote reduced words on the basis representing the finite set of elements $h_1, \ldots, h_k$. We may assume also that $h_i \neq h^{m}_j$ whenever $m \in \integers -\{0\}$ and $i \neq j$. Let $J=\langle a,b\rangle$
 where $a= h_1 \beta_1 \ldots h_k \beta_k, b= \beta'_1 h_1  \ldots \beta'_k h_k$, and where each $\beta_i$ and $\beta'_i$ is a reduced word on the basis satisfying that,
 for each $i \in \{1, \ldots, k\}$,

\begin{enumerate}
\item \label{item:pf1}$\beta_i \neq \beta^m_j$, $\beta'_i \neq (\beta'_j)^m$, and $\beta_i \neq (\beta'_j)^m$ whenever $i \neq j$ and $m \in \integers -\{0\}$,
\item \label{item:pf2} no $\beta_i, \beta'_i$ is a product of $h_i$'s or their inverses,
\item \label{item:pf3} for each $\beta_i$, its first letter is not equal to the last letter of $h_i$, and its last letter is not equal to the first letter of $h_{i+1}$ (modulo $k$). Similarly, for each $\beta'_i$,  its last letter is not equal to the first letter of $h_i$, and its first letter is not equal to the last letter of $h_{i+1}$ (modulo $k$). Moreover, the last letter of $\beta_k$ is not equal to the inverse of the last letter of $h_k$, the first letter of $\beta_1'$ is not equal to the inverse of the first letter of $h_1$, and the last letter of $\beta_k$ is not equal to the first letter of $\beta_1'$.
\end{enumerate}

 As there is no cancellation, $J$  is a rank-2 free group. By Lemma~\ref{lem:freemaln}, to prove that $J$ is malnormal it suffices to show that $J$ is isolated and malnormal on generators. The choice of the $\beta_i, \beta'_i$ implies in particular that $a$ and $b$ are not proper powers, and this implies in turn that $J$ is isolated, since $F$ is free. We now show that $J$ is malnormal on generators. Consider a conjugate $a^g=(h_1 \beta_1 \ldots h_k \beta_k)^g$ where $g \notin J$ (the case of $b^g$ is analogous). Since $g \notin J$, $g$ cannot be written as a non-trivial product of powers of $a$ and $b$ and their inverses. If $g$ cannot be written as a subword of a product of $a$'s  and $b$'s and their inverses, then $a^g$ cannot be an element of $J$ as there will be no cancellation. The choice of $a$ and $b$ implies that no cyclic permutation of $a$, $b$, their product or their proper powers lies in $J$, so no conjugate of $a$ by a subword of a product of $a$'s  and $b$'s and their inverses lies in $J$.  

Finally, consider a conjugate $h^g_i$ of $h_i$. If $\langle h^g_i \rangle\cap J$ is non-trivial, then since $F$ is free, it follows that $g$ must be a subword of some $j \in J$, and even in this case $h^g_i$ can only be a (non-trivial) cyclic permutation of of $a$, $b$, their product or their proper powers, but no such cyclic permutation is an element of $J$, so $J$ intersects all conjugates of $h_i$'s trivially. 
\end{proof}

\begin{proof}[Proof of Lemma~\ref{lem:cylics}]
It suffices to show that $G$ contains a malnormal quasiconvex free subgroup $J$ of arbitrarily high rank, for then if $J=\langle a_1, \ldots, a_k, b_1, \ldots, b_k\rangle$, the collection $\{H_1, \ldots, H_k\}$ where $H_i=\langle a_i,b_i \rangle$ is also malnormal and quasiconvex. For this, it suffices to show that $G$ contains a malnormal quasiconvex free group $J$ of \emph{some} rank $\geq 2$. Indeed, for any $n \in \integers$, $J$ contains infinitely many subgroups of rank $n$, all of which are malnormal and quasiconvex.

By Theorem~\ref{thm:quasi}, $G$ contains a free non-abelian quasiconvex subgroup $J_0$; we may assume further that $J_0$ has infinite-index in $G$. Let $\mathcal{J}_0$ be the lattice of infinite intersections of conjugates of $J_0$: 
this lattice is finite by Theorem~\ref{thm:finiteheight}. If $\mathcal{J}_0$ contains a non-abelian free group $J_1$ then we replace $J_0$ with $J_1$, and we can repeat this process a finite number of times until we either reach a maximal intersection of conjugates of some $J_i$ that is itself free non-abelian or until all subgroups in the lattice $\mathcal{J}_i$ are cyclic.
In the former case, the commensurator $C_G(J_i)$ is malnormal and quasiconvex by Lemma~\ref{lem:comms}. In the latter case, by Claim~\ref{clm:separatingcyclics}, $J_i$ contains a free non-abelian subgroup $J$ that forms a malnormal collection with each of these cyclic subgroups, hence $J$ is malnormal in $G$.
\end{proof}

Lemma~\ref{lem:cylics} can be improved to control intersections with quasiconvex subgroups: 

\begin{corollary}\label{cor:stabs}
Let $G$ be a non-elementary, torsion-free hyperbolic group and let $\{S_1, \ldots, S_\ell\}$  be a collection of quasiconvex subgroups of $G$. 
Then the collection  $\{H_1, \ldots, H_k\}$ from the conclusion of Lemma~\ref{lem:cylics} can be chosen so that $H_i \cap S_j$ is either trivial or equal to $H_i$ for each $i\in \{1, \ldots, k\}$ and each $j\in \{1, \ldots, \ell\}$.
\end{corollary}

\begin{remark}\label{rmk:stabs}
In particular, if $G$ is the fundamental group of a compact non-positively curved cube complex $X$, and $\mathcal{H}$ is the set of hyperplanes of $X$, then $\{H_1, \ldots, H_k\}$ can be chosen so that $H_i \cap Stab(\widetilde{U})$ is either trivial or equal to $H_i$ for each $U \in \mathcal{H}$ and each $i\in \{1, \ldots, k\}$. Indeed, since  $X$ is compact, it has finitely many hyperplanes, and hence finitely many hyperplane stabilisers. Each hyperplane stabiliser is quasi-isometrically embedded and  $\pi_1 X$ is hyperbolic, so each hyperplane stabiliser is quasiconvex. 
\end{remark}



\begin{proof}[Proof of Corollary~\ref{cor:stabs}]
It suffices to prove the result for a single malnormal, quasiconvex non-abelian subgroup $J\leq G$. Indeed, as explained in the first paragraph of the proof of Lemma~\ref{lem:cylics}, the subgroups in the malnormal collection $\{H_1, \ldots, H_k\}$ are produced as subgroups of a single non-abelian free subgroup $J$, so if we ensure that $J \cap S_j$ is either trivial or equal to $J$ for each $j \in \{1, \ldots, \ell\}$, then this will also be the case for $\{H_1, \ldots, H_k\}$.

We now proceed by induction on $\ell$. Assume that  $\ell=1$ and consider the intersection $J\cap S_1$, where $J$ is as provided in Lemma~\ref{lem:cylics}. If this intersection is trivial, then there is nothing to show, so suppose that $K_1:=J\cap S_1$ is non-trivial, so $K_1$ is either cyclic or free of rank $\geq 2$. If $K_1$ is cyclic, say generated by $k_1$, then as $J$ is free, by Claim~\ref{clm:separatingcyclics} there exists a $J'\leq J$ such that $\{J', \langle k_1\rangle\}$  is malnormal, so in particular $J' \cap S_1$ is trivial; if $K_1$ is free of rank $\geq 2$, then since $K_1$ is quasiconvex, Lemma~\ref{lem:cylics}  implies that there exists a quasiconvex, non-abelian free subgroup $J''\leq K_1$ that is malnormal in $G$, so that $J''\cap S_1=J''$. 

Now assume that the result holds for  $m=\ell-1$ and let $\{S_1, \ldots, S_\ell\}$ be a collection of quasiconvex free non-abelian subgroups.
Then by the induction hypothesis and Lemma~\ref{lem:cylics}, there is a quasiconvex non-abelian subgroup $J< G$ such that $J \cap S_i$ is trivial or equal to $J$ for each $i \in \{1, \ldots, m\}$. As before, if $K_\ell:=J \cap S_\ell=\{1\}$ then there is nothing to show, if $K_\ell$ is cyclic then by Claim~\ref{clm:separatingcyclics} there exists a $J'\leq J$ such that $J' \cap S_\ell$ is trivial, and since $J'\leq J$ then it is still the case that $J' \cap S_i$ is trivial or equal to $J'$ for each $i \in \{1, \ldots, m\}$.  
Finally, if $K_\ell$ is non-abelian then since it is quasiconvex, Lemma~\ref{lem:cylics} produces a new $J''$ inside $K_\ell$ for which $J'' \cap S_\ell=J''$. Since $J'' < J$, then $J''\cap S_i$ is trivial or equal to $J''$ for each $i \in \{1, \ldots, m\}$, and the result follows.
\end{proof}


\section{Main theorem}\label{sec:main}

\begin{theorem}\label{thm:main}
Let  $Q$ be a finitely presented group and $G$ be the fundamental group of a compact special cube complex $X$. If $G$ is hyperbolic and non-elementary, then there is a short exact sequence

\begin{equation}\label{eq:ses}
1 \rightarrow N \rightarrow \Gamma \rightarrow Q \rightarrow 1
\end{equation}

where \begin{enumerate}
\item $\Gamma$ is a hyperbolic, cocompactly cubulated group,
\item $N \cong G/K$  for some $K < G$,
\item \label{thmcc:3}$max\{cd(G),2\}\geq cd(\Gamma)\geq cd(G)-1$. In particular, $\Gamma$  is torsion-free.

\end{enumerate}
\end{theorem}

In what follows, we prove all parts of Theorem~\ref{thm:main} except for~\eqref{thmcc:3}, which we explain in the next section.

\begin{remark} Cocompact cubulability of $\Gamma$ hinges on the specialness of  $X$ as this is the hypothesis that is used in Theorem~\ref{thm:cyclic}. However, as noted in Remark~\ref{rmk:special}, in reality all that is needed is for the hyperplanes of $X$ to be embedded and 2-sided.
\end{remark}

\begin{proof}

Choose a finite presentation $\langle a_1, \ldots, a_s | r_1, \ldots, r_k \rangle$ for $Q$, let $B$ be a bouquet of $s$ circles $a_1, \ldots, a_s$ ,
and let $X$ be a compact non-positively curved cube complex  with $\pi_1 X =G =\langle x_1, \ldots, x_m \rangle$. 

By Lemma~\ref{lem:cylics} and Remark~\ref{rmk:stabs}, there is a malnormal collection $\{H_\ell\} \cup \{H'_{ij}\} \cup \{H''_{ij}\} $  of quasiconvex free subgroups of rank $\geq 2$ of $\pi_1 (X\vee B)$, so we can apply Theorem~\ref{thm:cyclic} to $X$ and $B$, where the $y_i$'s are given by:

 \begin{enumerate}
 \item $y_\ell= r_\ell $ for each $1 \leq \ell \leq k$, 
 \item $y'_{ij}= a_i x_ja_i^{-1}$ for each $1 \leq i \leq  s$,  $1 \leq j \leq m$, 
 \item $y''_{ij}= a_i^{-1} x_j a_i$ for each $1 \leq i \leq  s$,  $1 \leq j \leq m$, 
 \end{enumerate}

Hence, there are words $w_1, \ldots, w_\ell, w'_{1,1}, \ldots, w'_{ij}, w''_{1,1}, \ldots, w''_{ij} \in \pi_1 X$ for which the group $$\Gamma=G*\pi_1 B \big/ \langle\langle \{r_\ell w_\ell\}^k_{\ell=1}, \{a_i x_ja_i^{-1} w'_{ij}\}^{s,m}_{i=1,j=1}, \{a_i^{-1} x_j a_i w''_{ij}\}^{s,m}_{i=1,j=1} \rangle\rangle$$ is hyperbolic and acts properly and
cocompactly on a CAT(0) cube complex.

There is a homomorphism $\Gamma \overset{\phi} \longrightarrow \pi_1B$ that sends every generator of $\pi_1 X$ to $1$. Hence the set of relations $\{r_\ell w_\ell=1 \}_\ell$  map exactly to the relations $\{r_\ell=1\}_\ell$ in $\pi_1 B$, and so we see that the image of the homomorphism is $Q$.

The relations $\{ a_i x_ja_i^{-1} w'_{ij} =1 \}_{i \in S, j \in M},$ and $\{ a_i^{-1} x_j a_i w''_{ij}=1 \}_{i \in S, j \in M} $ ensure that $\langle x_1, \ldots, x_m \rangle$ is a normal subgroup of $\Gamma$ so $N=ker\phi=\langle x_1, \ldots, x_m \rangle$ and $N\cong \pi_1 X/K$ for some subgroup  $K < G$. 
\end{proof}

\section{Cohomological dimension}\label{sec:coho}

We briefly recall some standard facts
 about the cohomological dimension of groups. We refer the reader to~\cite{BrownBook82and94} for more details and proofs.

A \emph{resolution} for a module $M$ over a ring $R$ is a long exact sequence of $R-$modules
\[ \cdots \rightarrow M_n \rightarrow M_{n-1} \rightarrow \cdots \rightarrow M_1 \rightarrow M_0 \rightarrow M \rightarrow 0\]

A resolution is \emph{finite} if only finitely many of the $M_i$ are non-zero. The \emph{length} of a finite resolution is the maximum integer $n$ such that $M_n$ is non-zero.  A resolution is \emph{projective} if each $M_i$ is a projective module.

\begin{definition} The \emph{cohomological dimension} $cd(G)$ of a group $G$ is the length of the shortest projective resolution of $\integers$ as a trivial $\integers G$-module.
\end{definition}

There is a natural topological analogue of cohomological dimension:

\begin{definition} The \emph{geometric dimension} $gd(G)$ of a group $G$ is the least dimension of a classifying space for $G$.
\end{definition}

\begin{remark}
The cellular chain complex of a classifying space for a group $G$ yields a free (in particular, projective) resolution of $\integers$ over $\integers G$, the length of which is equal to the dimension of the classifying space. This implies immediately that $cd(G) \leq gd(G)$ for any group $G$. In particular, if $G$ is free, then $cd(G) = 1$.
\end{remark}

\begin{remark} The universal covers of non-positively curved cube complexes are CAT(0) spaces, and hence are contractible, so every non-positively curved cube complex $X$ is a classifying space for its fundamental group. Therefore, if $G=\pi_1 X$ for a compact non-positively curved cube complex $X$, then the cohomological dimension of $G$ is bounded above by the dimension of $X$.  
\end{remark}

\begin{proposition}The following hold for any group $G$:
\begin{enumerate}
\item If $G' < G$ then $cd(G') \leq cd (G)$
and equality holds provided that $cd (G) \leq \infty$ and $[G:G']< \infty.$
\item If  $1 \longrightarrow G' \longrightarrow G \longrightarrow G''  \longrightarrow 1$ is exact, then
$cd(G) \leq cd(G') + cd(G'').$
\item If $G=G_1 \ast G_2$, then $cd (G) = max\{cd (G_1), cd(G_2)\}.$
\end{enumerate}
\end{proposition}

The following result is a consequence of Proposition~\ref{prop:cohenlyndon}, stated below.

\begin{proposition} Let $G$, $Q$, $B$ and $\Gamma$ be as in Theorem~\ref{thm:main} and let $q:G \ast \pi_1B \rightarrow \Gamma$ be the natural quotient. Then $Ker(q)$ is free.
\end{proposition}

\begin{proposition} $\Gamma$ can be chosen so that $cd(\Gamma)\geq cd(G)-1$.
\end{proposition}

\begin{proof}
There is a short exact sequence 

$$1 \longrightarrow Ker(q) \longrightarrow G\ast \pi_1 B \longrightarrow \Gamma \longrightarrow 1.$$

Since $Ker(q)$ is a free group and $cd(G\ast \pi_1 B)=\max \{cd(G), 1\}=cd(G)$, then $cd(G)\leq cd(\Gamma) + cd(Ker(q))=cd(\Gamma)+1$.
\end{proof}

The torsion-freeness of $\Gamma$ will follow from having a finite upperbound on its cohomological dimension:

\begin{proposition}\label{prop:upperbound}
$\Gamma$ can be chosen so that  $max\{cd(G),2\}\geq cd(\Gamma)$
\end{proposition}

Before proving Proposition~\ref{prop:upperbound}, we state some auxiliary results.

\begin{definition}[The Cohen-Lyndon property]
Let $G$ be a group, $\{H_i\}_{i \in I}$ a family of subgroups and $N_i \lhd H_i$ for each $i \in I$. The triple $(G,\{H_i\},\{N_i\})$ has the \emph{Cohen-Lyndon property} if for each $i \in I$ there exists a left transversal $T_i$ of $H_i \langle \langle \cup_{i \in I} N_i \rangle \rangle$ in $G$ such that $\langle \langle \cup_{i \in I} N_i \rangle \rangle$ is the free product of the subgroups $N_i^t$ for $t \in T_i$, so $$\langle \langle \cup_{i \in I} N_i \rangle \rangle = \ast_{i \in I, t \in T_i}N_i^t.$$
\end{definition}

The Cohen-Lyndon property was first defined and studied in~\cite{CL63}, where it was proven to hold for triples $(F,C, \langle c \rangle)$ where $F$ is free, $C$ is a maximal cyclic subgroup of $F$, and $c \in C - \{1\}$. This was later generalised in~\cite{EH87} to the setting of free products of locally indicable groups. Most remarkably, it was recently proven in~\cite{Sun20} that triples $(G,\{H_i\},\{N_i\})$ have the Cohen-Lyndon property when the $H_i$ are ``hyperbolically embedded" subgroups of $G$ and the $N_i$ avoids a finite set of ``bad" elements depending only on the $H_i$. We will not define hyperbolically embedded subgroups here, and instead state only the particular case of the theorem that is required for our applications:

\begin{theorem}\cite{Sun20}
Let $G$ be hyperbolic, $\{H_i\}$ be malnormal and quasiconvex subgroups of $G$, and $N_i\lhd H_i$ for each $i$. Then there exists a finite set of elements $\{g_1, \ldots, g_n\} \in \cup_i H_i - \{1\}$ such that the triple $(G,\{H_i\},\{N_i\})$ has the Cohen-Lyndon property provided that $N_i \cap \{g_1, \ldots, g_n\}=\emptyset$ for all $i$.
\end{theorem}

To simplify notation, let $\{H_\ell\} \cup \{H'_{ij}\} \cup \{H''_{ij}\} := \mathbf{H}$ and write $H_{\iota} \in \mathbf{H}$. Let $S \subset \mathbf{H}$ be a finite set. It is clear from the constructions in Theorems~\ref{thm:main} and~\ref{thm:cyclic} (say, by applying Claim~\ref{clm:separatingcyclics} to $S$ before producing the cyclic subgroups in the proof of Theorem~\ref{thm:cyclic}) that the elements $\{r_\ell w_\ell\}^k_{\ell=1}, \{a_i x_ja_i^{-1} w'_{ij}\}^{s,m}_{i=1,j=1}, \{a_i^{-1} x_j a_i w''_{ij}\}^{s,m}_{i=1,j=1}$  can be chosen so that for each $c_{\iota} \in \{r_\ell w_\ell\}^k_{\ell=1} \cup \{a_i x_ja_i^{-1} w'_{ij}\}^{s,m}_{i=1,j=1} \cup \{a_i^{-1} x_j a_i w''_{ij}\}^{s,m}_{i=1,j=1}$, where $c_\iota \in H_\iota$, the intersection $\langle \langle c_{\iota} \rangle \rangle_{ H_\iota }\cap S$ is empty, and hence:

\begin{corollary}\label{prop:cohenlyndon}
$\Gamma$ can be chosen so that the triple $(G\ast \pi_1B, \{H_\iota\} , \{\langle \langle c_\iota \rangle \rangle_{H_\iota }\})$ has the Cohen-Lyndon property.
\end{corollary}

The following is proven in~\cite{PetSun21}:

\begin{proposition}\label{prop:boundscd}  If  $(G,\{H_i\},\{N_i\})$ has the Cohen-Lyndon property, then $$cd(G/\langle \langle \cup_{i \in I} N_i \rangle \rangle) \leq \max\{cd(G), \sup\{cd(H_i)+1\}, \sup \{cd(H_i/ N_i \}\}.$$
\end{proposition}

\begin{definition}A \emph{graphical} small cancellation presentation is a 1-dimensional cubical small cancellation presentation. Namely, a cubical presentation $X^*=\langle X | \{Y_i\}\rangle$ where $X$ is a graph and $Y_i \rightarrow X$ are graph immersions. 
\end{definition}

In the particular setting of graphical presentations, it is well-known that the coned-off space $X^*$ is aspherical. Concretely, the following Theorem holds. A proof is given in~\cite{Gruber15}, though we caution the reader that the language utilised there differs from that on this note.

\begin{theorem}
Let $X^*=\langle X | \{Y_i\}\rangle$ be a $C(6)$-graphical small cancellation presentation. Then $X^*$ is aspherical.
\end{theorem}

\begin{remark}
As in the statement of Corollary~\ref{prop:cohenlyndon}, we simplify the notation so that $\{H_\ell\} \cup \{H'_{ij}\} \cup \{H''_{ij}\} := \mathbf{H}$, $c_\iota \in H_\iota$ equals the corresponding $r_\ell w_\ell, a_i x_ja_i^{-1} w'_{ij}$, or $a_i^{-1} x_j a_i w''_{ij}$, and $C_\iota \rightarrow X \vee B$, $W_\iota \rightarrow X \vee B$ are the corresponding local isometries defined in the proof of Theorem~\ref{thm:cyclic} having $\pi_1(C_\iota)=H_\iota, \pi_1(W_\iota)=\langle c_\iota \rangle$ for each $\iota$.
\end{remark}

\begin{proof} [Proof of Proposition~\ref{prop:upperbound}]

By the Cohen-Lyndon property for $(G\ast \pi_1B, \{H_\iota\} , \{\langle \langle c_\iota \rangle \rangle_{H_\iota }\})$, then $cd(G\ast \pi_1B/\langle \langle \cup_\iota c_\iota \rangle \rangle) \leq \max \{cd(G)\ast \pi_1B, \sup \{cd(H_\iota)+1\}, \sup \{cd(H_\iota/ \langle \langle c_\iota \rangle \rangle_{ H_\iota }\})\}$, and since each $ H_\iota$ is free, then $cd(H_\iota)=1$ for all $H_\iota \in \mathbf{H}$. We claim that each of the quotients $H_\iota /  \langle \langle c_\iota \rangle \rangle_{H_\iota}$ has a $C(6)$-graphical small cancellation presentation, and so $cd( H_\iota / \langle \langle c_\iota \rangle \rangle_{H_\iota })\leq 2$ for all $H_\iota \in \mathbf{H}$.

To see this, consider the cubical presentation $\langle X\vee B | \{W_\iota \rightarrow X\vee B\}_\iota \rangle$ constructed in the proof of Theorem~\ref{thm:main}. As explained in the proof of Theorem~\ref{thm:cyclic}, each $H_\iota$ is carried by a local isometry $C_\iota \rightarrow X \vee B$. Since each $C_\iota$ is itself a non-positively curved cube complex and $\pi_1 C_\iota \cong H_\iota$ is free, then each $C_\iota$ is homotopy equivalent to a graph $\bar{C}_\iota$ in $C^{(1)}_\iota$. Similarly, each $W_\iota$ is homotopy equivalent to a cycle $\bar{W}_\iota$ in $W_\iota$ and we can further assume that each $\bar{W}_\iota $ lies in $\bar{C}_\iota$.
The intersections between pieces of each $\bar{W}_\iota$ are contained in the corresponding intersections between pieces of $W_\iota$, and the length of each $\bar{W}_\iota$ is bounded below by the diameter of the corresponding $W_\iota$, so each $\bar{W}_\iota$ has at least as many pieces as the corresponding $W_\iota$. Since the $W_\iota$ are chosen to satisfy at least $C'(\frac{1}{16})$, and in particular $C'(\frac{1}{16}) \Rightarrow C(6)$, then each $\langle \bar{C}_\iota | \bar{W}_\iota \rightarrow \bar{C}_\iota \rangle$ satisfies the $C(6)$ condition, and the proof is complete.
\end{proof}

This finishes the proof of Theorem~\ref{thm:main}.

\bibliographystyle{alpha}

\bibliography{bib9.bib}
\end{document}